\documentclass[12pt,a4paper,oneside,intlimits,sumlimits]{amsart}

\ifdefined\SMART
\usepackage[paperwidth=12cm,paperheight=24cm,left=5mm,right=5mm,top=10mm,bottom=5mm]{geometry}
\else
\calclayout
\fi

\usepackage{amssymb,latexsym,amsfonts,textcomp,fancyhdr,calc,graphicx}
\usepackage{pdfpages,amsfonts,amsthm,amssymb,latexsym,graphicx,leftidx,fancyhdr}
\usepackage{amsmath,amstext,amsthm,amssymb,amsxtra}
\usepackage[integrals]{wasysym}

\usepackage[allcolors=blue,allbordercolors=blue,pdfborderstyle={/S/U/W 1}]{hyperref}
\usepackage[ocgcolorlinks]{ocgx2}
\usepackage{dsfont}
\usepackage[framemethod=tikz]{mdframed}
\usepackage{asymptote}
\usepackage{enumerate}
\usepackage{changepage}

\usepackage[normalem]{ulem}
\usepackage{cite}

\usepackage{txfonts,pxfonts,tikz} 
\usepackage{tgschola}
\usepackage[T1]{fontenc}

\usepackage{mathtools}

\usepackage{bbm}

\theoremstyle{plain}
\newtheorem{theorem}{Theorem}

\newtheorem{corollary}{Corollary}
\newtheorem{proposition}{Proposition}

\theoremstyle{definition}

\newtheorem{case[theorem]}{Case}

\newcommand{\beql}[1]{\begin{equation}\label{#1}}
\newcommand{\eeq}{\end{equation}}

\newcommand{\Abs}[1]{{\left|{#1}\right|}}

\newcommand{\Set}[1]{{\left\{{#1}\right\}}}

\newcommand{\one}{{\bf 1}}

\newcommand{\RR}{{\mathbb R}}

\newcommand{\ZZ}{{\mathbb Z}}

\newcommand{\HHH}{{\mathcal H}}

\newcounter{rem}
\newcounter{step}
\newcounter{mysec}
\newcounter{mysubsec}[mysec]
\begin{document}
\sloppy

\title{Functions with comparable integrals on all $k$-planes}

\author{Mihail N. Kolountzakis}
\address{\href{http://math.uoc.gr/en/index.html}{Department of Mathematics and Applied Mathematics}, University of Crete,\\Voutes Campus, 70013 Heraklion, Greece,\newline and \newline \href{https://ics.forth.gr/}{Institute of Computer Science}, Foundation of Research and Technology Hellas, N. Plastira 100, Vassilika Vouton, 700 13, Heraklion, Greece}
\email{kolount@uoc.gr}

\author[G. E. Pfander]{G\"otz E. Pfander}
\address{Chair of Mathematics -- Scientific Computing, Catholic University of Eichst\"att--Ingolstadt, \href{https://www.ku.de/en/mids}{MIDS}, Auf der Schanz 49, 85049 Ingolstadt, Germany}
\email{pfander@ku.de}

\subjclass[2020]{44A12, 53C65, 52C22}

\keywords{Radon transform, k-plane transform, affine Grassmannian, integral geometry, Steinhaus tiling problem, measurable tilings, lattice tilings.}

\begin{abstract}
Let $d \ge 2$ and $1 \le k \le d-1$.
We show that if $f:\RR^d\to\RR$ is nonnegative and measurable and $0<m\le M <\infty$ then it is impossible that on almost all affine $k$-planes $P$ in $\RR^d$ the integral of $f$ on $P$ lies between $m$ and $M$. Let $G = \ZZ^k \times \Set{0}^{d-k}$. Using $m=M=1$ and $f$ being the indicator function of a measurable set in $\RR^d$ this then implies that there is no measurable Steinhaus set for the group $G$ in $\RR^d$. In other words there is no measurable set $S \subseteq \RR^d$ such that $S$ tiles $\RR^d$ with $T(G)$, for all $T \in O(d)$. We also generalize our impossiblity results concerning the size of line-integrals of functions to integrals on strips in the plane.
\end{abstract}

\date{\today}

\maketitle

\tableofcontents

\section{Introduction}\label{s:intro}

A classical question  of Steinhaus \cite{sierpinski1958probleme,moser1981research} asks if there is a set $S \subseteq \RR^2$ such that for all $\theta\in[0, 2\pi)$ $S$ tiles $\RR^2$ when translated at the locations $R_\theta\ZZ^2$, where $R_\theta$ denotes rotation by $\theta$ around the origin. This means that every $x \in \RR^2$ should belong to exactly one translate $S+n$, for $n\in R_\theta\ZZ^2$. In a major result Jackson and Mauldin \cite{jackson2002sets,jackson2002lattice,mauldin2001some,jackson2003survey} proved the existence of lattice Steinhaus sets in the plane (not necessarily measurable).

We are interested in the version of this problem where, on the one hand, we demand that $S$ is Lebesgue measurable and, on the other, we view the tiling $S+R_\theta\ZZ^2$ as an almost everywhere property: for all $\theta$, almost every $x \in \RR^2$ belongs to exactly one translate $S+n$, where $n \in R_\theta\ZZ^2$. It is easy to see that if $S$ is a measurable Steinhaus set then its measure must be 1. Measurable Steinhaus sets are not known to exist when tiling by $\ZZ^2$ and its rotates.

Regarding the measurable version of the problem Sierpinski \cite{sierpinski1958probleme} showed that a bounded set which is either closed or open cannot have the
Steinhaus property while Croft \cite{croft1982three} and Beck \cite{beck1989lattice} showed that this is also impossible for any bounded and measurable set (see also \cite{mallinikova1995}). Kolountzakis \cite{kolountzakis1996problem,kolountzakis1996new} and Kolountzakis and Wolff \cite{kolountzakis1999steinhaus} proved that any measurable set in the plane that has the measurable Steinhaus property must necessarily have very slow decay at infinity. In \cite{kolountzakis1999steinhaus} it was also shown that there can be no measurable Steinhaus sets in dimension $d\ge 3$ (here we are seeking a measurable subset of $\RR^d$ which tiles with every rotate $T(\ZZ^d)$, where $T \in O(d)$). This was reproved later by Kolountzakis and Papadimitrakis \cite{kolountzakis2002steinhaus} by a very different method. See also \cite{chan2007steinhaus,mauldin2003comments,ciucu1996remark,srivastava2005steinhaus}. Kolountzakis \cite{kolountzakis1997multi} looks at the case where the linear transformation $\rho$ is only required to take on finitely many values and the question is studied there for when finitely many lattices of the same volume in $\RR^d$ can have the same fundamental domain. Especially for the case of two lattices there have been some major results recently \cite{grepstad2026bounded,spyridakis2026bounded}.

The Steinhaus tiling problem makes sense for more general subsets of $\RR^2$ or $\RR^d$. If $\Gamma \subseteq \RR^d$ we say that the measurable $S \subseteq \RR^d$ has the Steinhaus property with respect to $\Gamma$ if $S+T(\Gamma)$ is a tiling for all $T \in O(d)$. It was shown in \cite{kolountzakis2017measurable} that for any finite $\Gamma\subseteq\RR^d$ there is no measurable set in $\RR^d$ with the Steinhaus property. 

In \cite{kolountzakis2017measurable} it was also shown that if $\Gamma = \ZZ\times\Set{0} \subseteq \RR^2$ then there is also no measurable set in $\RR^2$ with the Steinhaus property. It is this last result that concerns us in this paper. It was shown in \cite{kolountzakis2017measurable} that if such a set $S$ existed (in the plane) and $f=\one_S$ then $\int_L f = 1$ on almost all lines $L$ in the plane. This was then proved to be impossible, thus disproving the existence of sets $S$ with the Steinhaus property with respect to $\ZZ\times\Set{0}$ in the plane.

The essential part of the proof of this fact was to prove that it is impossible for a set $S$ in the plane with indicator function $f$ to satisfy
\begin{equation}\label{all-lines}
m \le \int_L f \le M,\ \ \ \text{for almost all lines $L$},
\end{equation}
for some constants $0 < m \le M < \infty$. (For the application to Steinhaus sets it was enough to have $m = M = 1$.) This was proved using the theory of harmonic functions. The proof in \cite{kolountzakis2017measurable} introduces, for $S\subset\RR^2$, the function in $\RR^3$
\begin{equation}\label{potential}
F(z)=\int_{\RR^2}\frac{\one_S(w)}{|z-w|}\,dw, \qquad z\in\RR^3,
\end{equation}
shows that its boundary values on $\RR^2$ are constant, proves that $F$ is continuous and harmonic in the upper half-space, and then uses the Poisson representation for bounded harmonic functions to obtain a contradiction with the decay of $F$ in the vertical direction.

The main contribution of this paper is to give an alternative, elementary proof of the impossibility of \eqref{all-lines}, while at the same time generalizing from indicator functions to all nonnegative functions (a modification of the proof in \cite{kolountzakis2017measurable} can also be made to work for nonnegative functions) and generalizing also from line integrals to integrals over affine $k$-planes in $\RR^d$.

There is some related literature in integral geometry and tomography. For instance Ilmavirta and Paternain \cite{ilmavirta2019functions} prove that a smooth bounded strictly convex Euclidean domain admits an $L^1$ density integrating to one over almost every line meeting the domain if and only if the domain is a ball. Armitage \cite{armitage1994non} and Zalcman \cite{zalcman1982uniqueness} construct signed functions on the plane whose integrals along all lines are 0. What distinguishes our problem is the absence of integrability or support conditions on the function.

Thus our main result, whose proof is presented in Section \ref{s:proof}, is:

\begin{theorem}\label{th:main}
Let $f \ge 0$ be a nonnegative Lebesgue measurable function on $\RR^d$, $d\ge 2$, and let $1 \le k \le d-1$. Then there are no constants $0 < m \le M < \infty$ for which
\begin{equation}\label{k-planes}
m \le \int_P f \, d\HHH^k \le M,\ \ \ \text{for almost all $k$-planes $P$}.
\end{equation}
\end{theorem}

In the following proposition we show that if $S$ has the Steinhaus property with repsect to $\ZZ^k \times \Set{0}^{d-k}$ in $\RR^d$ then its indicator function would have to satisfy \eqref{k-planes}.

\begin{proposition}\label{steinhaus-to-planes}
Let $1\le k < d$ and $G=\ZZ^k\times\{0\}^{d-k}\subseteq\RR^d$, let
$S\subset\RR^d$ be measurable, and write $f=\one_S$.
Assume that for almost every $T\in O(d)$,
$\sum_{n\in T(G)}f(x-n)=1$ for almost every $x\in\RR^d$.
Then, for almost every $x\in\RR^d$, one has
$\HHH^k(S\cap P)=1$ for almost every affine $k$-plane $P$
through $x$.
\end{proposition}

\begin{proof}
Put $E_0=\RR^k\times\Set{0}^{d-k}$. Fix a rotation $T$ for which
the hypothesis holds, and set $E=T(E_0)$ and $\Lambda=T(G)$.
Writing $\RR^d=E^\perp\oplus E$ and applying Fubini, for almost
every $y\in E^\perp$ we have
$\sum_{n\in\Lambda}f(y+u-n)=1$ for almost every $u\in E$.
Integrating this for $u$ in a fundamental domain $Q$ of $\Lambda$ in $E$ gives
$\HHH^k(S\cap(y+E)) = 1$. Hence, for almost every $T$, we have
$\HHH^k(S\cap(x+T(E_0)))=1$ for almost every $x\in\RR^d$. Using Fubini again we obtain that for a.e.\ $x$, this equality holds for a.e.\ $T\in O(d)$, which is the desired conclusion.

\end{proof}

Proposition \ref{steinhaus-to-planes} and Theorem \ref{th:main} immediately imply
\begin{corollary}\label{cor:no-steinhaus-sets}
If $d \ge 2$ and $1\le k \le d-1$ then there are no measurable sets with the Steinhaus property with respect to $\ZZ^k \times \Set{0}^{d-k}$.
\end{corollary}

If the nonnegative, measurable function $f:\RR^2\to\RR$ has line integrals between $m$ and $M$ along almost all straight lines of the plane then for almost all directions $u$ and almost every strip $S$ in the plane, parallel to $u$, of width $R$ the integral $\int_S f$ lies in $[mR, MR]$. We know there are no functions with line integrals in $[m, M]$ but are there any functions whose integrals over any strip $S$ of width $R$ satisfy $mR \le \int_S f \le M R$? The following corollary shows this is not the case.

\begin{corollary}\label{cor:strips-all-widths}
For any $0 < m \le M <\infty$ there is no nonnegative, measurable function $f$ on $\RR^2$ such that if $S$ is any strip and $R$ is its width then $mR \le \int_S f \le M R$.
\end{corollary}

\begin{proof}
Fix a unit vector $u$ and, for $t \in \RR$, let $L_t=\Set{x: x \cdot u = t}$. For almost every $t$ put $F_u(t) = \int_{L_t} f$. By Fubini, for every $a \in \RR$ and $R>0$,
\begin{equation}
m \le \frac1R\int_a^{a+R} F_u(t) \,dt \le M. \label{averaged-line-bounds}
\end{equation}
In particular $F_u \in L^1_{\mathrm{loc}}(\RR)$. By the Lebesgue differentiation theorem, letting $R \to 0$ in \eqref{averaged-line-bounds} gives
\begin{equation}
m \le F_u(t) \le M \label{line-bounds}
\end{equation}
for a.e.\ $t$. Since this holds for every direction $u$, it follows that \eqref{line-bounds} holds for almost every line in the plane, contradicting Theorem \ref{th:main}.

\end{proof}

Key in the proof of Corollary \ref{cor:strips-all-widths} was the fact that we could use strips of arbitrarily small width. It turns out that this is not necessary.

\begin{corollary}\label{cor:strips-of-width-one}
For any $0 < m \le M <\infty$ there is no nonnegative, measurable function $f$ on $\RR^2$ such that if $S$ is any strip of width $1$ then $m \le \int_S f \le M$.
\end{corollary}

\begin{proof}

Let $\phi(x)=\frac1\pi e^{-|x|^2}$, $x \in \RR^2$, and $g = f*\phi$. The convolution is well defined because of the decay of the gaussian and the assumption on the integrals of $f$ on strips of width 1.

Fix a unit vector $u$ and define the measure $\mu_u$ on $\RR$ that arises by projecting of $f$ onto $\RR u$: 
$$
\mu_u(I) = \int_{\Set{x: x\cdot u\in I}} f(x) \, dx.\ \ (I \subseteq \RR)
$$
For every interval $I$ of length $1$ we have then $m \le \mu_u(I) \le M$.

The one-dimensional marginal of $\phi$ is $h(t)=\frac1{\sqrt{\pi}}e^{-t^2}$. Consequently the projection of $g(x)$ onto $\RR u$ has density
\begin{equation}
q_u(t)=\int_{\mathbb R}h(t-s)\,d\mu_u(s).
\label{smoothed-projection}
\end{equation}
From \eqref{smoothed-projection} we easily obtain constants $m_0, M_0 > 0$, depending only on $m$ and $M$, such that
\begin{equation}\label{q-bounds}
m_0\le q_u(t)\le M_0 \ \ \ \text{ for every } u, t.
\end{equation}

If $S$ is now any strip of arbitrary width $R>0$, corresponding under $x\mapsto x\cdot u$ to an interval $J$ of length $R$, then $\int_S g=\int_J q_u(t)\,dt$ and \eqref{q-bounds} gives
$$
m_0 R \le \int_S g \le M_0 R.
$$
This contradicts Corollary \ref{cor:strips-all-widths}.

\end{proof}

\section{Proof of Theorem \ref{th:main}}\label{s:proof}

Let $f \ge 0$ satisfy \eqref{k-planes}. We may assume, translating $f$, that almost all $k$-planes $P$ through the origin satisfy $m \le \int_P f\,d\HHH^k \le M$.Then all rotates $f(T\cdot)$, with $T \in O(d)$, satisfy \eqref{k-planes}, and so does their average $\int_{O(d)} f(T\cdot)\,dT$, so we may assume that $f(x) = f(r)$, $r=\Abs{x}$, is radial. Integrating $f$ along any non-exceptional $k$-plane through the origin shows that 
\begin{equation}\label{f-integrable}
\int_0^\infty f(r) r^{k-1}\,dr < \infty.
\end{equation}

Write $F(t)$ for the integral of $f$ along any $k$-plane at distance $t$ from the origin (since $f$ is radial this does not depend on the chosen plane, only on $t$). We obtain
\begin{equation}\label{F}
F(t) = \sigma_{k-1} \int_t^\infty f(r) r (r^2-t^2)^{\frac{k-2}{2}}\,dr
\end{equation}
and for a.e.\ $t$ we have $m \le F(t) \le M$. (Here $\sigma_\ell$ denotes the surface area of a the unit sphere in $\RR^{\ell+1}$.)

For $k\ge 2$ we have from \eqref{F}
$$
F(t) \le \sigma_{k-1} \int_t^\infty f(r) r^{k-1}\,dr,
$$
hence $\lim_{t\to\infty}F(t) = 0$ from \eqref{f-integrable}, which contradicts the assumed lower bound $m$ for $F(t)$.

For $k=1$ equation \eqref{F} becomes
\begin{equation}\label{F1}
F(t) = 2 \int_t^\infty f(r) r (r^2-t^2)^{-\frac{1}{2}}\,dr.
\end{equation}
We will show that
$$
\lim_{R\to\infty} \frac1R \int_R^{2R} F(t)\,dt = 0,
$$
thus arriving again at contradiction with $F(t) \ge m$.

We have (Fubini)
\begin{align}
\frac1R \int_R^{2R} F(t)\,dt &= \frac{2}{R}\int_R^\infty f(r) r \left( \int_R^{\min\Set{2R,r}} \frac{dt}{\sqrt{r^2-t^2}} \right) \,dr \nonumber\\
 &\le \frac{2}{R}\int_R^\infty f(r) \sqrt{r} \left( \int_R^{\min\Set{2R,r}} \frac{dt}{\sqrt{r-t}} \right) \,dr. \label{to-bound}
\end{align}
We show that for $r \ge R$
\begin{equation}\label{OR}
S(r, R) := \sqrt{r} \int_R^{\min\Set{2R,r}} \frac{dt}{\sqrt{r-t}} = 2\sqrt{r}(\sqrt{r-R}-\sqrt{r-\min\Set{2R,r}}) \le C R
\end{equation}
for some constant $C$.

If $r \ge 2R$ then
\begin{align*}
S(r, R) &= 2\sqrt{r}(\sqrt{r-R}-\sqrt{r-2R}) \\
 &= \frac{2R\sqrt{r}}{\sqrt{r-R}+\sqrt{r-2R}} \\
 &\le \frac{2R\sqrt{r}}{\sqrt{r/2}} \\
 &\le C R.
\end{align*}

If $r \le 2R$ then
$$
S(r,R) = 2\sqrt{r}\sqrt{r-R} \le 2\sqrt{2R} \sqrt{2R-R} \le C R,
$$
which completes the proof of \eqref{OR}.

If follows that
$$
\frac1R\int_R^{2R} F(t)\,dt \le C \int_R^\infty f(r) \,dr
$$
which tends to $0$ as $R\to\infty$ from \eqref{f-integrable}, a contradiction with $F(t) \ge m$.

This completes the proof of Theorem \ref{th:main}.

\noindent{\bf Acknowledgement.} The authors used ChatGPT for mathematical exploration and literature search. Some key arguments were suggested by ChatGPT. The authors have written and checked every statement in this paper.

\bibliographystyle{alpha}
\bibliography{mk-bibliography.bib}

\begin{thebibliography}{Kol96b}

\bibitem[Arm94]{armitage1994non}
D.~H. Armitage.
\newblock A non-constant continuous function on the plane whose integral on
  every line is zero.
\newblock {\em The American Mathematical Monthly}, 101(9):892--894, 1994.

\bibitem[Bec89]{beck1989lattice}
J.~Beck.
\newblock {On a lattice point problem of H. Steinhaus}.
\newblock {\em Studia Sci. Math. Hung}, 24:263--268, 1989.

\bibitem[Ciu96]{ciucu1996remark}
M.~Ciucu.
\newblock {A remark on sets having the Steinhaus property}.
\newblock {\em Combinatorica}, 16(3):321--324, 1996.

\bibitem[CM07]{chan2007steinhaus}
W.~K. Chan and R.~Mauldin.
\newblock {Steinhaus tiling problem and integral quadratic forms}.
\newblock {\em Proc. Am. Math. Soc.}, 135(2):337--342, 2007.

\bibitem[Cro82]{croft1982three}
H.~T. Croft.
\newblock {Three lattice-point problems of Steinhaus}.
\newblock {\em Q. J. Math.}, 33(1):71--83, 1982.

\bibitem[GK26]{grepstad2026bounded}
S.~Grepstad and M.~N. Kolountzakis.
\newblock Bounded common fundamental domains for two lattices.
\newblock {\em Advances in Mathematics}, 487:110776, 2026.

\bibitem[IP19]{ilmavirta2019functions}
J.~Ilmavirta and G.~P. Paternain.
\newblock Functions of constant geodesic x-ray transform.
\newblock {\em Inverse Problems}, 35(6):065002, 2019.

\bibitem[JM02a]{jackson2002lattice}
S.~Jackson and R.~Mauldin.
\newblock {On a lattice problem of H. Steinhaus}.
\newblock {\em J. Am. Math. Soc.}, 15(4):817--856, 2002.

\bibitem[JM02b]{jackson2002sets}
S.~Jackson and R.~D. Mauldin.
\newblock {Sets meeting isometric copies of the lattice $\ZZ^2$ in exactly one
  point}.
\newblock {\em Proc. Natl. Acad. Sci. USA}, 99(25):15883--15887, 2002.

\bibitem[JM03]{jackson2003survey}
S.~Jackson and R.~D. Mauldin.
\newblock {Survey of the Steinhaus tiling problem}.
\newblock {\em Bull. Symb. Log.}, 9(03):335--361, 2003.

\bibitem[Kol96a]{kolountzakis1996new}
M.~N. Kolountzakis.
\newblock {A new estimate for a problem of Steinhaus}.
\newblock {\em Int. Math. Res. Not.}, 1996(11):547--555, 1996.

\bibitem[Kol96b]{kolountzakis1996problem}
M.~N. Kolountzakis.
\newblock {A problem of Steinhaus: Can all placements of a planar set contain
  exactly one lattice point?}
\newblock {\em Progress in Mathematics}, 139:559--566, 1996.

\bibitem[Kol97]{kolountzakis1997multi}
M.~N. Kolountzakis.
\newblock Multi-lattice tiles.
\newblock {\em Int. Math. Res. Not.}, 1997(19):937--952, 1997.

\bibitem[KP02]{kolountzakis2002steinhaus}
M.~N. Kolountzakis and M.~Papadimitrakis.
\newblock {The Steinhaus tiling problem and the range of certain quadratic
  forms}.
\newblock {\em Ill. J. Math.}, 46(3):947--951, 2002.

\bibitem[KP17]{kolountzakis2017measurable}
M.~N. Kolountzakis and M.~Papadimitrakis.
\newblock {Measurable Steinhaus sets do not exist for finite sets or the
  integers in the plane}.
\newblock {\em Bull. Lond. Math. Soc.}, 49(5):798--805, 2017.

\bibitem[KW99]{kolountzakis1999steinhaus}
M.~N. Kolountzakis and T.~H. Wolff.
\newblock {On the Steinhaus tiling problem}.
\newblock {\em Mathematika}, 46(02):253--280, 1999.

\bibitem[Mau]{mauldin2001some}
R.~Mauldin.
\newblock {Some problems in set theory, analysis and geometry, Paul Erd{\"o}s
  and his mathematics, I (Budapest, 1999), 493--506}.
\newblock {\em Bolyai Soc. Math. Stud}, 11.

\bibitem[Mos81]{moser1981research}
W.~Moser.
\newblock {\em {Research problems in discrete geometry}}.
\newblock Department of Mathematics, McGill University, 1981.

\bibitem[MR95]{mallinikova1995}
E.~Mallinikova and S.~Rukshin.
\newblock {On one Steinhaus problem}.
\newblock {\em Vestn. St. Petersbg. Univ., Math.}, 28(1):28--32, 1995.

\bibitem[MY03]{mauldin2003comments}
R.~Mauldin and A.~Yingst.
\newblock {Comments about the Steinhaus tiling problem}.
\newblock {\em Proc. Am. Math. Soc.}, 131(7):2071--2079, 2003.

\bibitem[Sie58]{sierpinski1958probleme}
W.~Sierpi{\'n}ski.
\newblock {Sur un probleme de H. Steinhaus concernant les ensembles de points
  sur le plan}.
\newblock {\em Fundam. Math.}, 2(46):191--194, 1958.

\bibitem[Spy26]{spyridakis2026bounded}
E.~Spyridakis.
\newblock Bounded and measurable common fundamental domains for two lattices.
\newblock {\em arXiv preprint arXiv:2602.23480}, 2026.

\bibitem[ST05]{srivastava2005steinhaus}
S.~M. Srivastava and R.~Thangadurai.
\newblock {On Steinhaus sets}.
\newblock {\em Expo. Math.}, 23(2):171--177, 2005.

\bibitem[Zal82]{zalcman1982uniqueness}
L.~Zalcman.
\newblock {Uniqueness and nonuniqueness for the Radon transform}.
\newblock {\em Bulletin of the London Mathematical Society}, 14(3):241--245,
  1982.

\end{thebibliography}

\end{document}